\documentclass[12pt]{amsart}

\usepackage{amssymb}

\usepackage{enumerate}

\usepackage{graphicx}

\makeatletter
\@namedef{subjclassname@2010}{%
  \textup{2010} Mathematics Subject Classification}
\makeatother
\usepackage{tikz}
\newcommand{\hooklong}{\lhook\joinrel\joinrel\joinrel{-}}
\newtheorem{thm}{Theorem}[section]
\newtheorem{cor}[thm]{Corollary}
\newtheorem{lem}[thm]{Lemma}

\theoremstyle{definition}

\newtheorem{rem}[thm]{Remark}
\newtheorem{exa}[thm]{Example}

\def\cgcoft{\cal{A}_{CG}^{co\ ft}}
\def\cgcoa{\cal{A}_{CG}^{co}}
\def\bb#1{{\mathbb{#1}}}
\def\cal#1{{\mathcal{#1}}}
\def\Hom{{\rm{Hom}}}
\def\lim{{\rm{lim}}}

\def\deg{{\rm{deg}}}

\def\mfrac#1{{\mathfrak{#1}}}

\title[Inverse and Antipode]{Cogroups in the Category of Connected Graded Algebras whose inverse and antipode coincide}

\author{Hiroshi Kihara}

\address{%
Center for Mathematical Sciences, University of Aizu\\
Tsuruga, Ikki-machi, Aizu-wakamatsu City, Fukushima, 965-8580, Japan.}

\email{kihara@u-aizu.ac.jp}

\thanks{This work was completed with the support of our \TeX-pert.}



\begin{document}
\begin{abstract}
Let $A$ be a cogroup in the category of connected graded algebras over a commutative ring $R$. Let $\nu$ denote the inverse of $A$ and $\chi$ the antipode of the underlying Hopf algebra of $A$. We clarify the differences and similarities of $ \nu $ and $ \chi $, and show that $\nu$ coincides with $\chi$ if and only if $A$ is commutative as a graded algebra. Let $\cgcoa$ be the category of cogroups satisfying these equivalent conditions. If $R$ is a field, the category $\cgcoa$ is completely determined. We also establish an equivalence of the full subcategory of $\cgcoa$ consisting of objects of finite type with a full subcategory of the category of positively graded $R$-modules without any assumption on $R$. The results in the case of $R=\bb{Q}$ are applied to the theory of co-$H$-groups.
\end{abstract}
\maketitle
\section{Introduction and main results.}
Cogroups in the category of connected graded algebras are important both in algebra and topology, and they have been studied by many authors (e.g. [6], [9], [3], [5]). In this paper, particular concern is focused on two important, but confusing self-maps of such a cogroup $A$: the inverse $\nu$ of $A$ and the antipode $\chi$ of the underlying Hopf algebra of $A$. We investigate the differences and similarities between the inverse $\nu$ and the antipode $\chi$, and show that $\nu$ and $\chi$ coincide if and only if $A$ is commutative as a graded algebra. We denote the category of cogroups satisfying these equivalent conditions by $\cgcoa$. Then we establish an equivalence of the category $\cgcoa$ with a full subcategory of the category $\cal{M}$ of positively graded modules in the case where the ground ring $R$ is a field. In this case, the category $\cgcoa$ can be completely determined since the full subcategory of $\cal{M}$ is simple. If we restrict ourselves to objects of finite type, we obtain a similar equivalence of categories without any assumption on $R$. However, the full subcategory of $\cal{M}$, and hence $\cgcoa$ is rather complicated in the case where $R$ is not a field, which is illustrated by examples. We also show that a cogroup $A$ whose underlying algebra is graded commutative is a cocommutative cogroup at least if $R$ is a field or $A$ is of finite type. The results for $R=\bb{Q}$ are applied to the theory of co-$H$-groups.
\par\indent Let us begin by recalling fundamental results on cogroups in the category of connected graded algebras.
\par\indent
Let $\cal{A}$ be the category of connected graded algebras over a commutative ring $R$, and $\cal{C}$ the category of connected graded coalgebras over $R$ (cf. [8]). Let $\cal{A}_{CM}$ and $\cal{A}_{CG}$ denote the subcategory of comonoids in $\cal{A}$ and that of cogroups in $\cal{A}$ respectively; comonoids and cogroups in $\cal{A}$ are with respect to the categorical coproduct in $\cal{A}$.
\par\indent
For $C\in \cal{C}$ and $B\in \cal{A}, G(C,B)$ denotes the set of homomorphisms of graded $R$-modules $f:C\longrightarrow B$ such that $f_0$ is the identity homomorphism of $R$. $G(C,B)$ is a group under the convolution product $\ast$ (see Lemma \ref{gcbgroup}).
\par\indent
The following theorem is essentially due to Berstein [6, Theorem 1.2].
\vspace{2mm}
\par\noindent {\bf{Theorem A.}} (1) There exist functors $S:\cal{A}_{CM}\longrightarrow \cal{C}$ and $T_{CM}:\cal{C}\longrightarrow \cal{A}_{CM}$ such that $S$ and $T_{CM}$ are mutually inverses up to natural isomorphism. For $C\in \cal{C}$, the underlying graded algebra of $T_{CM}(C)$ is a tensor algebra on $\overline{C}=\underset{n>0}\bigoplus\ C_n$.
\par\noindent (2) For $A\in \cal{A}_{CM}$, the monoid-valued functor $\cal{A}(A,\ \ )$ is naturally isomorphic to the functor $G(SA,\ \ )$.
\par\noindent (3) $\cal{A}_{CG}=\cal{A}_{CM}$ holds.

\begin{rem}
Berstein's proof that any comonoid in $\cal{A}$ is a cogroup is incorrect ([6, p.262]); the inverse of a cogroup in $\cal{A}$ must be a morphism of $\cal{A}$. Thus we assert part 2, which easily implies part 3 (the proof of part 2 is given in Section 2). See [5, Theorem 34.24] for another proof of part 3.
\end{rem}
For $A\in \cal{A}_{CG}$, let $\nu$ and $\chi$ denote the inverse of $A$ and the antipode of the underlying Hopf algebra of $A$ respectively ([6, p. 261], Remark \ref{remof2.2comonoidA}). These two maps often cause confusion; the fault of Berstein's proof also comes from this confusion. Thus it is important to clarify the differences and similarities between $\nu$ and $\chi$.
\par\indent
A major difference is that $\nu$ is a homomorphism of graded algebras, whereas $\chi$ is an antihomomorphism of graded algebras ([8, 8.7. Proposition]). On the other hand, $\nu$ is very similar to $\chi$ in view of the following theorem.
\begin{thm}\label{underHopf.alge}
Let $A$ be a cogroup in $\cal{A}$. Then the inverse $\nu$ of $A$ and the antipode $\chi$ of the underlying Hopf algebra of $A$ coincide on $SA$.
\end{thm}
Using Theorem \ref{underHopf.alge}, we find necessary and sufficient conditions that $\nu$ coincides with $\chi$.
\begin{thm}\label{nuequaltochi}
Let $A$ be a cogroup in $\cal{A}$. Let $\nu$ be the inverse of $A$ and $\chi$ the antipode of the underlying Hopf algebra of $A$. Then the following are equivalent:
\begin{itemize}
\item[(i)] $\nu=\chi$.
\item[(ii)] $\chi$ is a homomorphism of graded algebras.
\item[(iii)] $A$ is commutative as a graded algebra.\\
\end{itemize}
\if\else
\begin{itemize}
\item[(2)]
\begin{itemize}
\item[(a)] If the ground ring $R$ is a field, then the conditions {\rm{(i)-(iii)}} are equivalent to the follwing:
\begin{itemize}
\item[(iv)] $SA$ is isomorphic to $R$ or $C(R,n)$ for some $n$, where $n$ is even if $chR\neq 2$.
\end{itemize}
\item[(b)] If $A$ is of finite type as a graded module, then the conditions {\rm{(i)-(iii)}} are equivalent to the follwing:
\begin{itemize}
\item[(v)] For any maximal ideal $\mfrac{m}$ of $R$, the localization $SA_{\mfrac{m}}$ at $\mfrac{m}$ is isomorphic to $R_{\mfrac{m}}$ or $C(R_{\mfrac{m}}/{\mfrac{a}}_{\mfrac{m}},n)$ for some ideal ${\mfrac{a}}\subset {\mfrac{m}}$ and some $n$, where $n$ is even if $chR_{\mfrac{m}}/{\mfrac{a}}_{\mfrac{m}}\neq 2$.
\end{itemize}
\item[(c)] Suppose that $chR=2$. If $R$ is a field or $A$ is of finite type as a graded module, then the conditions {\rm{(i)-(iii)}} are equivalent to the following:
\begin{itemize}
\item[(vi)] $\nu=\chi=1_A$.
\end{itemize}
\end{itemize}
\end{itemize}
\fi
\end{thm}
\indent From Theorem 1.3, cogroups in $\cal{A}$ which we are concerned with are just cogroups whose underlying algebras are graded commutative. Thus we study the full subcategory $\cgcoa$ of $\cal{A}_{CG}$ consisting of such cogroups, and the full subcategory $\cgcoft$ of $\cgcoa$ consisting of objects whose underlying graded modules are of finite type.
\par\indent Let us construct the functor $C$ from the category $\cal{M}$ of positively graded $R$-modules to the category $\cal{C}$ which is used together with the equivalence $T_{CM}:\cal{C}\longrightarrow \cal{A}_{CG}$ to investigate the category $\cgcoa$. For an $R$-module $M$, let $M[n]$ denote the graded $R$-module with $M[n]_n=M$ and $M[n]_i=0$ ($i\neq n$). For a positively graded $R$-module $N$, let $C(N)$ denote the graded coalgebra which is defined to be the graded module $R[0]\oplus N$ equipped with the trivial coproduct (see Lemma 2.4); in the case where $N=M[n]$, $C(N)$ is often denoted by $C(M,n)$. This construction defines a fully faithful functor from $\cal{M}$ to $\cal{C}$.
\par\indent A positively graded module $N$ is called locally at most singly generated if for any maximal ideal $\mfrac{m}$ of $R$, the localization $N_{\mfrac{m}}$ at $\mfrac{m}$ is isomorphic to 0 or $R_{\mfrac{m}}/\mfrac{a}_{\mfrac{m}}[n]$ for some ideal $\mfrac{a}\subset \mfrac{m}$ and some $n$, where $n$ is even if $chR_{\mfrac{m}}/\mfrac{a}_{\mfrac{m}}\neq 2$. Let $\cal{S}$ denote the full subcategory of $\cal{M}$ consisting of locally at most singly generated objects. $\cal{S}^{ft}$ denotes the full subcategory of $\cal{S}$ consisting of objects of finite type.

\par\indent Let $\cal{F}$ denote the full subcategory of $\cal{S}$ consisting of objects isomorphic to one of $0,R[1],R[2],\cdots$ (resp. $0,R[2],R[4]\cdots$) if $chR=2$ (resp. $chR\neq 2$). Note that $\cal{F}$ coincides with $\cal{S}$ in the case where $R$ is a field.
\par\indent We can investigate the structure of categories $\cgcoa$ and $\cgcoft$ via the composite functor $\cal{M}\xrightarrow[]{\ C\ } \cal{C} \xrightarrow[]{\ T_{CM}\ } \cal{A}_{CG}$.
\begin{thm}
\begin{itemize}
\item[(1)] $T_{CM}\circ C$ restricts to the fully faithful functor $T_{CM}\circ C:\cal{S}\longrightarrow \cgcoa$. 
\item[(2)] If the ground ring $R$ is a field, then $T_{CM}\circ C$ restricts to an equicalence of categories $T_{CM}\circ C:\cal{F}\longrightarrow \cgcoa$.
\item[(3)] $T_{CM}\circ C$ restricts to an equicalence of categories $T_{CM}\circ C: \cal{S}^{ft}\longrightarrow \cgcoft$.
\end{itemize}
\end{thm}
Let us explicitly describe the most basic objects of $\cgcoa: T_{CM}C(0)=T_{CM}R$ and $T_{CM}C(R[n])=T_{CM}C(R,n)$
\par\indent  $A(0)$ denote the graded algebra $R$ equipped with the obvious cogroup structure. For $n>0$, let $A(n)$ denote the cogroup in $\cal{A}$ satisfying the following conditions:
\begin{itemize}
\item[(1)] The underlying graded algebra of $A(n)$ is the polynomial ring $R[X]$ on a generator $X$ of degree $n$.
\item[(2)] The comultiplication $\Phi:R[X]\longrightarrow R[X]\ast R[X]$ and the inverse $\nu:R[X]\longrightarrow R[X]$ are the homomorphisms of graded algebras determined by $\Phi(X)=X'+X''$ and $\nu(X)=-X$ respectively, where $X'$ and $X''$ are the images of $X$ under the canonical inclusions to the first and second functors of $R[X]\ast R[X]$ respectively.
\end{itemize}
The cogroups $A(0)$ and $A(n)$ give explicit descriptions of $T_{CM}R$ and $T_{CM}C(R,n)$ respectively (cf. Lemma 2.4).
\par\indent From Theorem 1.4(2), we can easily deduce
\begin{cor}
Suppose that the ground ring $R$ is a field.
\begin{itemize}
\item[(1)] The category $\cgcoa$ is equivalent to the full subcategory consisting of $A(0),A(1),A(2),\cdots$ (resp. $A(0),A(2),A(4),\cdots$) if $chR=2$ (resp. $chR\neq 2$).
\item[(2)] $\cgcoa(A(m),A(n))= \begin{cases} R & (m=n>0) \\ 0 & (otherwise).\end{cases} $
\end{itemize}
\end{cor}
\par\indent
In the case where $R$ is not a field, there are objects of $\cgcoa$ other than $A(n)$; Theorem 1.4(1) implies that $A=T_{CM}C(R/\mfrac{a}, 2k)$ is an object of $\cgcoa$. In the case where the ground ring $R$ is not local, we can construct several types of objects $A$ in $\cgcoa$ such that $A$ is not even singly generated as a graded algebra.
 \begin{exa}
 \begin{itemize}
  \item[(1)] Let $M$ be a projective module of rank 1, and $n$ a positive integer. Suppose that $n$ is even if $chR\neq 2$. Then $A=T_{CM}C(M,n)$ is an object in $\cgcoa$.
 \item[(2)] Suppose that a family $\{R/\mfrac{a}_{\alpha}[n_{\alpha}]\}_{\alpha\in\Lambda}$ in $\cal{M}$ satisfies the following conditions:
 \begin{itemize}
 \item[(a)] $\mfrac{a}_{\alpha}+\mfrac{a}_{\beta}=R$ if $\alpha\neq\beta$.
 \item[(b)] $n_{\alpha}$ is even if $chR/\mfrac{a}_{\alpha}\neq 2$.
 \end{itemize}
 (If $\{\mfrac{a}_{\alpha}\}_{\alpha\in\Lambda}$ is a family of pairwise distinct maximal ideals, the condition (a) is automatically satisfied.) Set $N=\underset{\alpha\in\Lambda}\bigoplus R/\mfrac{a}_{\alpha}[n_{\alpha}]$. Then $A=T_{CM}C(N)$ is an object in $\cgcoa$.

 \item[(3)] Let $S$ be a multiplicatively closed subset of $R$. Then if $T_{CM}C(M,n)(=T_{CM}C_R(M,n))$ is commutative as a graded algebra, so is $T_{CM}C_R(S^{-1} M,n)$. In particular, $A=T_{CM}C_{\bb{Z}}(T,2k)$ is a cogroup in $\cal{A}$ such that $A$ is in $\cgcoa$ and $A$ is not of finite type for any ring $T$ with $\bb{Z}\subsetneqq T\subset \bb{Q}$.
 \end{itemize}
 \end{exa}
 \if\else
 \begin{rem}
 Suppose that $chR=2$. Then the cogroups $A(n)$ and those in Example 1.4 give examples of graded Hopf algebras whose antipode is the identity, by Theorem 1.3(2)(c). (Though a cogroup in Example 1.5 may not be of finite type as a graded module, it can be seen that $\nu=\chi=1_A$.)
 \end{rem}
\fi
\par\indent
We have investigated the category $\cgcoa$ of cogroups in $\cal{A}$ whose underlying algebras are graded commutative. Berstein found the condition that a cogroup $A$ is cocommutative and the condition that a cogroup $A$ is cocommutative as a graded coalgebra ([6, p.262]). Let $\cal{A}_{coCG}$ be the full subcategory of $\cal{A}_{CG}$ consisting of cocommutative objects, and $^{co}\cal{A}_{CG}$ the full subcategory of $\cal{A}_{CG}$ consisting of objects whose underlying coalgebras are graded cocommutative. We would like to compare the full subcategories $\cal{A}_{CG}^{co}$, $\cal{A}_{coCG}$ and $^{co}\cal{A}_{CG}$.
\par\indent
Let $^{co}\cal{C}$ denote the full subcategory of $\cal{C}$ consisting of graded cocommutative objects. Note that the functor $C:\cal{M}\longrightarrow\cal{C}$ factors as a composite $\cal{M}\xrightarrow[]{\ C\ }\ ^{co}\cal{C}\hooklong\hspace{-2mm}\longrightarrow\cal{C}$ of two fully faithful functors.
\par\indent
The following corollary summarizes our results on $\cgcoa$ and the relationship of $\cgcoa$, $\cal{A}_{coCG}$ and $^{co}\cal{A}_{CG}$.
\begin{cor}
\begin{itemize}
\item[(1)] The functors $T_{CM}$ and $C$ define the vertical functors of the diagram
\begin{center}
\newcommand{\xycoordinateofnode}[2]{%
	\node at (#1) {#2};
}
\newcommand{\directarrowintikz}[2]{\draw [->] (#1) -- (#2);}
\newcommand{\symarrowintikz}[2]{\draw [<->] (#1) -- (#2);}
\newcommand{\lineattikz}[2]{\draw [-] (#1) -- (#2);}
\newcommand{\attachhook}[3]{%
	\node at (#1) [#2] {\reflectbox{\rotatebox[origin=c]{#3}{$\hooklong$}}};
}
\newcommand{\attacnume}[3]{%
	\node at (#1) [#2] {#3};
}
\begin{tikzpicture}
\xycoordinateofnode{0.6,-2}{$\hooklong$} \xycoordinateofnode{0.5,-0}{$\hooklong$}
\xycoordinateofnode{2.6,-2}{$\hooklong$} \xycoordinateofnode{2.6,-0}{$\hooklong$}
\xycoordinateofnode{4.7,-2}{$\hooklong$} \xycoordinateofnode{4.7,-0}{$\hooklong$}
\xycoordinateofnode{6.7,-2}{$\hooklong$} \xycoordinateofnode{6.7,-0}{$\hooklong$}
\directarrowintikz{0,-0.25}{0,-1.75} \directarrowintikz{2,-0.25}{2,-1.75} 
\directarrowintikz{4,-0.25}{4,-1.75} \directarrowintikz{6,-0.25}{6,-1.75}
\directarrowintikz{8,-0.25}{8,-1.75}
\attacnume{0,-1}{left}{{\small{$T_{CM}\circ C$}}} \attacnume{2,-1}{left}{{\small{$T_{CM}\circ C$}}} 
\attacnume{4,-1}{left}{{\small{$T_{CM}\circ C$}}} \attacnume{6,-1}{left}{{\small{$T_{CM}$}}} \attacnume{8,-1}{left}{{\small{$T_{CM}$}}} 
\directarrowintikz{0.45,-0}{1.6,-0} \directarrowintikz{2.5,-0}{3.4,-0}
\directarrowintikz{4.6,-0}{5.4,-0} \directarrowintikz{6.6,-0}{7.5,-0}
\directarrowintikz{0.45,-2}{1.4,-2} \directarrowintikz{2.5,-2}{3.4,-2}
\directarrowintikz{4.6,-2}{5.4,-2} \directarrowintikz{6.6,-2}{7.5,-2}
\xycoordinateofnode{5,0.25}{{\small{$C$}}}
\xycoordinateofnode{0,0}{$\cal{S}^{ft}$}
\xycoordinateofnode{2,0}{$\cal{S}$}
\xycoordinateofnode{4,0}{$\cal{M}$}
\xycoordinateofnode{6,0}{$^{co}\cal{C}$}
\xycoordinateofnode{8,0}{$\cal{C}$}
\xycoordinateofnode{0,-2}{$\cgcoft$}
\xycoordinateofnode{2,-2}{$\cal{A}_{coCG}^{co}$}
\xycoordinateofnode{4,-2}{$\cal{A}_{coCG}$}
\xycoordinateofnode{6,-2}{$^{co}\cal{A}_{CG}$}
\xycoordinateofnode{8,-2}{$\cal{A}_{CG},$}
\end{tikzpicture}
\end{center}
where $\cal{A}_{coCG}^{co}$ denotes the full subcategory of $\cal{A}_{CG}$ consisting of objects which are in both $\cgcoa$ and $\cal{A}_{coCG}$.
\item[(2)] The diagram above is commutative up to natural isomorphism, and all the functors are fully faithful.
\item[(3)] All the vertical functors except the one with source $\cal{S}$ are equivalences of categories.
\item[(4)] If the ground ring $R$ is a field, then $\cal{F}=\cal{S}^{ft}=\cal{S}$ and $\cgcoft=\cal{A}_{coCG}^{co}=\cal{A}^{co}_{CG}$ hold.
\end{itemize}
\end{cor}
\begin{rem}
The fully faithful functors $\cal{S}^{ft}\hooklong\hspace{-2mm}\longrightarrow\cal{S}$ and $\cgcoft\hooklong\hspace{-2mm}\longrightarrow\cal{A}_{coCG}^{co}$ are not equivalences of categories in general (cf. Example 1.6).
\end{rem}
\indent
Next we summarize the relationship between cogroups in $\cal{A}$ and co-$H$-groups, and give a topological application of Theorems \ref{nuequaltochi} and 1.4.
\par\indent
Let $\cal{H}'$ be the category of $1$-connected co-$H$-groups having the homotopy type of a $CW$-complex, and homotopy classes of co-$H$-maps. Let $\cal{H}'_{\bb{Q}}$ be the full subcategory of $\cal{H'}$ consisting of $1$-connected rational co-$H$-groups.
\par\indent The following is deduced from results of Berstein [6] and Scheerer [9].
\vspace{2mm}\par\noindent {\bf{Theorem B.}} Suppose that the ground ring $R$ is a field.
\begin{itemize}
\item[(1)] The loop space homology $H_{\ast}(\Omega\ \cdot\ ; \ R)$ defines a functor from $\cal{H}'$ to $\cal{A}_{CG}$.
\item[(2)] Let $X$ be an object of $\cal{H}'$. Then the cogroup $H_{\ast}(\Omega X ; R)$ in $\cal{A}$ has $H_{\ast}(\Omega\nu ;  R)$ and $H_{\ast}(\chi ;  R)$ as an inverse and an antipode respectively, where $\nu$ is a homotopy inverse of the co-$H$-group $X$ and $\chi$ is a homotopy inverse of the $H$-group $\Omega X$.
\item[(3)] If $R$ is the rationals $\bb{Q}$, $H_{\ast}(\Omega\ \cdot\ ; R)$ restricts to a fully faithful functor from $\cal{H}_{\bb{Q}}'$ to $\cal{A}_{CG}$.
\end{itemize}

\begin{rem}\label{assertiononnu}
\par\noindent
(1) Let $\cal{H}'^{ft}_{\bb{Q}}$ be the full subcategory of $\cal{H}_{\bb{Q}}'$ consisting of objects of finite type, and let $^{co}\cal{A}_{CG}^{ft}$ be the full subcategory of $^{co}\cal{A}_{CG}$ consisting of objects of finite type. Then $H_{\ast}(\Omega\ \cdot\ ;\bb{Q})$ restricts to an equivalence of $\cal{H}'^{ft}_{\bb{Q}}$ and $^{co}\cal{A}_{CG}^{ft}$ (cf. [1, Theorem A$'$], Theorem A and [6, p. 262]).
\par\noindent
(2) In view of Theorem B, the fact that any comonoid in $\cal{A}$ is a cogroup (Theorem A(2)) corresponds to the fact that any $1$-connected homotopy-associative co-$H$-complex is a co-$H$-group. On the other hand, the fact that any connected bialgebra has an antipode (Lemma 2.1 and Remark 2.2) corresponds to the fact that any connected homotopy-associative $H$-complex is an $H$-group. (Refer to [2, pp. 1147-1148].)
\end{rem}
Theorem B is used to deduce the following result on co-$H$-groups from Theorems 1.3 and 1.4.
\begin{thm}\label{nusamequaltochi}
Let $X$ be a $1$-connected rational co-$H$-group having the homotopy type of a $CW$-complex. Let $\nu$ be a homotopy inverse of $X$ and $\chi$ a homotopy inverse of $\Omega X$. Then the following are equivalent:
\begin{itemize}
\item[(i)] $\Omega\nu\simeq\chi$.
\item[(ii)] $\chi$ is an $H$-map.
\item[(iii)] $\Omega X$ is a homotopy commutative $H$-group.
\item[(iv)] $X$ is co-$H$-equivalent to a singleton or the rationalization of $S^{2n+1}(=\Sigma S^{2n})$ for some $n>0$.
\item[(v)] $X$ is rationally contractible or rationally equivalent to $S^{2n+1}$ for some $n>0$.
\end{itemize}
\end{thm}
We give proofs of the results in Section 2.
\section{Proofs of main results.}
We begin by recalling a result of Milnor-Moore [8]. The following lemma is Proposition 8.2 in [8]; we need not only the statement but also the proof. (Another reason why we record even its proof is that the inductive formula for the convolution-inverse of $f$ in [8] is incorrect; it is the formula for the antipode.)

\begin{lem}\label{gcbgroup}
Let $C=(C,\Delta,\epsilon)$ and $B=(B,\mu,\eta)$ be objects of $\cal{C}$ and $\cal{A}$ respectively. Then $G(C,B)$ is a group under the convolution product with identity $C \xrightarrow[]{\ \epsilon\ } R \xrightarrow[]{\ \eta\ } B$.
\begin{proof}
Recall that for $f,g\in G(C,B)$, the convolution product $f\ast g$ is defined to be the composite
$$
C\xrightarrow[]{\ \Delta\ } C\otimes C \xrightarrow[]{\ f\otimes g\ } B\otimes B \xrightarrow[]{\ \mu\ } B.
$$
Then it is clear that $G(C,B)$ is a monoid with identity $\eta\epsilon$.
\par\indent For $f\in G(C,B)$, we construct a right inverse $g$ by induction. Set $g_0=1_R$ and suppose that $g$ is defined on $C_{<n}:=\underset{i<n}\bigoplus\ C_i$. For $x\in C_n$, write $\Delta x=x\otimes 1 +\Sigma\ y_i\otimes z_i +1\otimes x$ with $0<\deg\ y_i, \deg\ z_i<n$. Then, by solving $(\mu\circ(f\otimes g)\circ\Delta)x=0$, we obtain
$$
g(x)=-f(x)-\Sigma\ f(y_i)g(z_i).
$$
Similarly we can construct a left inverse $h$ of $f$. The usual argument implies that $g$ coincides with $h$, and hence that $g$ is an inverse of $f$.
\end{proof}
\end{lem}
\begin{proof}[{\it{Proof of Theorem A(2).}}]
Consider the natural bijection $\cal{A}(A,B)\longrightarrow G(SA,B)$ induced by the inclusion $SA\hooklong\hspace{-2mm}\longrightarrow A$ (cf. Theorem A(1) and the construction in [6]). Then formula (2.6$'$) in [6] implies that this bijection preserves multiplication.
\end{proof}
\begin{rem}\label{remof2.2comonoidA}
Let $A$ be a comonoid in $\cal{A}$. Then the existence of an inverse $\nu$ of $A$ and that of an antipode $\chi$ of the underlying bialgebra of $A$ follow from Lemma \ref{gcbgroup}. The antipode $\chi$ is the convolution-inverse of $1_A\in G(A,A)$, and the inverse $\nu$ is the homomorphism of algebras corresponding to the convolution-inverse of the canonical inclusion $i\in G(SA,A)$ via the natural isomorphism of Theorem A(2). 
\end{rem}
\begin{proof}[Proof of Theorem \ref{underHopf.alge}]
Since $SA$ is a subcoalgebra of the underlying coalgebra of $A$ ([6, Lemma 2.5]), the canonical inclusion $i:SA\hooklong\hspace{-3.5mm}\longrightarrow  A$ induces a monoid homomorphism $i^{\sharp}: G(A,A)\hooklong\hspace{-3.5mm}\longrightarrow G(SA,A)$. Since $ \chi $ is the convolution-inverse of $1_A$, $i^{\sharp} \chi = \chi|_{SA}$ is the convolution-inverse of $i$,and $ \nu $ is just the extension of $\chi|_{SA}:SA\longrightarrow A$ as a graded algebra homomorphism (cf. Theorem A(1) and Remark 2.2).
\end{proof}
For the proof of Theorem \ref{nuequaltochi}, we prove the following lemma.
\begin{lem}\label{con.gra.Hopf}
Let $H$ be a connected graded Hopf algebra over $R$. Then the antipode $\chi$ of $H$ is surjective.
\begin{proof}
We proceed by induction on the degree.
\par\indent It is obvious that $\chi$ is bijective on $H_0$. Suppose that $\chi$ is surjective in ${\rm{degree}}<n$. Since $\chi$ is an antihomomorphism of graded algebras ([8, 8.7 Proposition]), $\chi$ restricts to an endomorpism of decomposable elements in $H_n$. (Decomposable elements in $H_n$ are sums of products of elements of ${\rm{degree}}<n$.) Thus $\chi$ induces the endomorphism $\overline{\chi_n}$ on $Q_nH:= H_n/\{$decomposable elements in $H_n\}$. Since the endomorphism of decomposable elements in $H_n$ is surjective by induction hypothesis, it is enough to show that the endomorphism $\overline{\chi_n}$ on $Q_nH$ is surjective.
 For $x\in H_n$, write $\Delta x= x\otimes 1+\Sigma y_i\otimes z_i +1\otimes x$ with $0<\deg\ y_i,\deg\ z_i<n$, and recall the inductive formula for the convolution-inverse in the proof of Lemma \ref{gcbgroup}. Then the inductive formula applied to the case of $f=1_H$ gives
$$\chi (x)=-x-\Sigma y_i \chi (z_i),$$
which implies that $\overline{\chi_n}$ is the multiplication by $-1$, and hence that $\overline{\chi_n}$ is surjective.
\end{proof}
\end{lem}
\begin{proof}[\it{Proof of Theorem \ref{nuequaltochi}}]
(i) $\Rightarrow$ (ii) Obvious.\\
(ii) $\Rightarrow$ (iii) Let $a$ and $b$ be elements of $A_p$ and $A_q$ respectively. By Lemma \ref{con.gra.Hopf}, we can choose elements $\tilde{a}$ and $\tilde{b}$ with $\chi(\tilde{a})=a$ and $\chi(\tilde{b})=b$. Note that $\chi$ is both an antihomomorphism and a homomorphism of graded algebras by [8, 8.7 Proposition] and the assumption. Then we have
$$ab=(-1)^{pq}\chi(\tilde{b}\tilde{a})=(-1)^{pq}ba.$$
(iii) $\Rightarrow$ (i) Since $A$ is commutative as a graded algebra, $\chi$ is a homomorphism of graded algebras. Thus Theorem A(1) and Theorem \ref{underHopf.alge} imply that $\nu=\chi$.
\if\else
\item[(2)]
\begin{itemize}
\item[(a)] It is enough to show that (iii) is equivalent to (iv).\\
(iii) $\Leftarrow$ (iv) By Lemma 2.5(1). \\
(iii) $\Rightarrow$ (iv) Since the underlying graded algebra of $A$ is isomorphic to $T(\overline{SA})$ (Theorem A(1)), $\overline{SA}$ has at most one generator of degree $n>0$. Thus the implication follows from Lemma 2.5(1).
\item[(b)] We prove the equivalence of (iii) and (v). By Lemma 2.5(2), we may assume that $R$ is local.\\
(iii) $\Leftarrow$ (v) By Lemma 2.5(1). \\
(iii) $\Rightarrow$ (v) By the definition of the tensor algebra, the graded commutativity of $A\cong T_{CM}(SA)\cong T(\overline{SA})$ implies the graded commutativity of $A\underset{R}{\otimes} \kappa(\mfrac{m})\cong T_{\kappa(\mfrac{m})}(\overline{SA}\underset{R}{\otimes}\kappa(\mfrac{m}))$, where $\kappa(\mfrac{m})$ denotes the residue field of the local ring $R=(R,\mfrac{m})$. Thus $\overline{SA}$ is zero or singly generated by part a and Nakayama's lemma ([Matsumura, pp, 8-9]). Hence $SA$ is isomorphic to $R$ or $C(R/\mfrac{a},n)$ for some $\mfrac{a}\subset \mfrac{m}$ and some $n$, where $n$ is even if $chR/\mfrac{a}\neq 2$.
\item[(c)] If $R$ is a field, the result follows from part (a) and Lemma 2.4(3). For the case where $A$ is of finite type, note that the localization $\chi_{\mfrac{m}}$ is an antipode of $A_{\mfrac{m}}$. Then the result follows from part (b), Lemma 2.4(3) and [A-M, Propositions 3.3 and 3.8].
\end{itemize}
\end{itemize}
\fi
\end{proof}
For the proof of Theorem 1.4, we prove the following lemmas. The categorical coproduct in $\cal{A}$ is denoted by $\ast$.
\begin{lem}
For a positively graded module $N$, let $C(N)$ denote the graded module $R[0]\oplus N$ equipped with the coproduct defined by
$$
\Delta(1)=1 \otimes 1\ for\ 1\in R=C(N)_0,
$$
$$
\Delta(x)=x \otimes 1 + 1 \otimes x\ for\ x\in N=\overline{C(N)}.
$$
\begin{itemize}
\item[(1)] $C(N)$ is a graded coalgebra. If $N=M[n]$, this structure is a unique graded coalgebra structure on $R[0]\oplus N$.
\item[(2)] Let $B$ be a connected graded algebra. Then the group $G(C(N),B)$ is naturally isomorphic to the module $\Hom_R(N,B)$ of homomorphisms of graded $R$-modules.
\item[(3)] $T_{CM}C(N)$ is a cogroup in $\cal{A}$ whose underlying graded algebra is the tensor algebra $T(N)$. The comultiplication $\Phi: T(N)\longrightarrow T(N)\ast T(N)$ is the homomorphism of graded algebras determined by $\Phi(x)=x'+x''$ for $x\in N$, where $x'$ and $x''$ are the images of $x$ under the canonical inclusions to the first and second factors of $T(N)\ast T(N)$ respectively. The inverse $\nu: T(N)\longrightarrow T(N)$ is the homomorphism of graded algebras determined by $\nu(x)=-x$ for $x\in N$.
\end{itemize}
\begin{proof} 
\begin{itemize}
\item[(1)] It is easily seen that $C(N)$ satisfies the coassociativity and counit axioms. The assertion for $N=M[n]$ is obvious.
\item[(2)] By the definition, there is a natural bijection from $G(C(N),B)$ to $\Hom_R(N,B)$. It is easily verified that it preserves multiplication by the definition of the convolution product.
\item[(3)] By Theorem A(1), $T_{CM}C(N)=T(N)$. The formula on $\Phi$ follows from that on $\Delta$ and formula $(2,6')$ in [6]. The formula on $\nu$ follows from that on $\Phi$.
\end{itemize}
\vspace{-4mm}
\end{proof}
\end{lem}
\if\else
\begin{cor}
$T_{CM}C(R,n)$ is a cogroup in $\cal{A}$ whose underlying graded algebra is the polynomial algebra $R[X]$ on a generator $X$ of degree $n$. The comultiplication $\Phi:R[X]\longrightarrow R[X]\ast R[X]$ is the homomorphism of graded algebras determined by $\Phi(X)=X'+X''$, and the inverse $\nu:R[X]\longrightarrow R[X]$ is the homomorphism of graded algebras determined by $\nu(x)=-X$.
\end{cor}
\fi
\begin{lem}
Let $N$ be a graded $R$-module.
\begin{itemize}
\item[(1)] Suppose that $N=R/\mfrac{a}[n]$ for some $n>0$ and some $\mfrac{a}\subsetneqq R$. Then the tensor algebra $T(N)$ is commutative as a graded algebra if and only if $n$ is even or $ch(R/\mfrac{a})=2$.
\item[(2)] $T(N)=T_R(N)$ is commutative as a graded algebra if and only if $T_{R_\mfrac{m}}(N_\mfrac{m})$ is commutative as a graded algebra for any maximal ideal $\mfrac{m}$ of $R$.
\end{itemize}
\begin{proof}
\begin{itemize}
\item[(1)] Let $x$ be a generator of $N_n$ and write $N_n=R/\mfrac{a}\cdot x$. Then we have
\begin{eqnarray}
T(N) &=& R+R/\mfrac{a}\cdot x + R/\mfrac{a}\cdot x\underset{R}{\otimes} R/\mfrac{a}\cdot x +\cdots \nonumber \\
     &=& R+R/\mfrac{a}\cdot x + R/\mfrac{a}\cdot x^2 +\cdots \nonumber,
\end{eqnarray}
which shows that $T(N)$ is commutative as an ungraded algebra. Thus the graded commutativity of $T(N)$ is equivalent to the condition that if $n$ is odd, then $x\cdot x=-x\cdot x$. The above expansion of $T(N)$ shows that $2\cdot x^2=0$ if and only if $chR/\mfrac{a}=2$, which completes the proof of part 1.
\item[(2)] ($\Rightarrow$) Note that localizations commute with tensor products and direct sums. Then the implication is obvious from the definition of the tensor algebra.\\
($\Leftarrow$) The implication follows from [4, Propositions 3.3 and 3.8].
\end{itemize}
\vspace{-4mm}
\end{proof}
\end{lem}
\begin{lem}
Let $N$ be a positively graded $R$-module.
\begin{itemize}
\item[(1)] Suppose that $R$ is a field. Then the tensor algebra $T(N)$ is graded commutative if and only if $N$ is in $\cal{F}$.
\item[(2)] Suppose that $N$ is of finite type. Then tensor algebra $T(N)$ is graded commutative if and only if $N$ is in $\cal{S}^{ft}$.
\end{itemize}
\begin{proof}
\begin{itemize}
\item[(1)] ($\Leftarrow$) By Lemma 2.5(1).\\
($\Rightarrow$) We can easily see that $N$ is zero or singly geenerated. Thus the implication follows from Lemma 2.5(1).
\item[(2)] By Lemma 2.5(2), we may assume that $R$ is local.\\
($\Leftarrow$) By Lemma 2.5(1).\\
($\Rightarrow$) By the definition of the tensor algebra, the graded commutativity of $T(N)=T_R(N)$ implies the graded commutativity of $T_R(N)\underset{R}{\otimes}\kappa(\mfrac{m})=T_{\kappa(\mfrac{m})}(N\underset{R}{\otimes}\kappa(\mfrac{m}))$, where $\kappa(\mfrac{m})$ denotes the residue field of the local ring $R=(R,\mfrac{m})$. Thus $N$ is zero or singly generated by part 1 and Nakayama's lemma ([7, pp.8-9]). Hence $N$ is in $\cal{S}^{ft}$ by Lemma 2.5(1).
\end{itemize}
\vspace{-4mm}
\end{proof}
\end{lem}
\begin{proof}[Proof of Theorem 1.4]
\begin{itemize}
\item[(1)] Since $C:\cal{M} \longrightarrow\cal{C}$ and $T_{CM}:\cal{C}\longrightarrow\cal{A}_{CG}$ are fully faithful, so is $T_{CM}\circ C$. Thus the assertion follows from Lemma 2.5.
\item[(2)] Since $T_{CM}\circ C:\cal{M}\longrightarrow \cal{A}_{CG}$ is fully faithful, it is enough to show that for a cogroup $A$ in $\cal{A}$, the condition (iii) in Theorem 1.3 is equivalent to the following:
\vspace{1.5mm}\begin{itemize}
\item[(iv)] $A$ is isomorphic to $T_{CM}C(N)$ for some $N\in\cal{F}$.
\end{itemize}\vspace{1.5mm}
Since $A$ is isomorphic to $T(\overline{SA})$ in $\cal{A}$ (Theorem A(1)), (iii) is equivalent to the condition that $\overline{SA}$ is in $\cal{F}$ (Lemma 2.6(1)), which is equivalent to (iv) by Lemma 2.4(1).
\item[(3)] Note that a positively graded module $N$ is of finite type if and only if so is the tensor algebra $T(N)$. Thus we show that for a cogroup $A$ in $\cal{A}$ whose underlying graded module is of finite type, the condition (iii) in Theorem 1.3 is equivalent to the following:
\vspace{1.5mm}\begin{itemize}
\item[(v)] $A$ is isomorphic to $T_{CM}C(N)$ for some $N\in\cal{S}^{ft}$. 
\end{itemize} \vspace{1.5mm}
Since $A$ is isomorphic to $T(\overline{SA})$ in $\cal{A}$ (Theorem A(1)), (iii) is equivalent to the condition that $\overline{SA}$ is in $\cal{S}^{ft}$ (Lemma 2.6(2)), which is equivalent to (v) by Lemma 2.4(1) and an argument similar to that of the proof of Lemma 2.5(2). This equivalence and the full-faithfulness of $T_{CM}\circ C:\cal{M}\longrightarrow\cal{A}_{CG}$ complete the proof.
\end{itemize}
\vspace{-4mm}
\end{proof}
\begin{proof}[Proof of Corollary 1.5]
It is immediate from Theorem 1.4(2).
\end{proof}
\begin{rem}
Note that $A=T_{CM}C(N)$ ($N\in\cal{S}$) satisfies the equiation $\nu=\chi=1_A$ in the case where $chR=2$ (cf. Lemma 2.4.(3)).
\end{rem}

\begin{proof}[Proof of Example 1.6]
\begin{itemize}
\item[(1)] Since $M[n]$ is in $\cal{S}$, $A=T_{CM}C(M,n)$ is in $\cgcoa$ by Theorem 1.4(1).
\item[(2)] It is easily seen that for any maximal ideal $\mfrac{m}$, $N_{\mfrac{m}}=0$ or $N_{\mfrac{m}}=R_{\mfrac{m}}/\mfrac{a}_{\alpha\mfrac{m}}[n_{\alpha}]$ for some $\alpha$. Thus $A$ is commutative as a graded algebra by Theorem 1.4(1).
\item[(3)] Note that $B\in\cal{A}$ is graded commutative if and only if so is the nonunital graded algebra $\overline{B}=\underset{i>0}\bigoplus B_i$. Note also that $S^{-1}$ commutes with tensor products and direct sums. Then the result easily follows.
\end{itemize}
\vspace{-1.5mm}\end{proof}
\begin{proof}[Proof of Corollary 1.7]
\begin{itemize}
\item[(1)] Note that for $A\in\cal{A}_{CG}$, $A$ is in $\cal{A}_{coCG}$ if and only if $SA$ is isomorphic to $C(\overline{SA})$ ([6, Corollary 2.6]). Note also that for $A\in\cal{A}_{CG}$, $A$ is in $^{co}\cal{A}_{CG}$ if and only if $SA$ is cocommutative as a graded coalgebra ([6, p. 262]). Then we see that the functor $T_{CM}: \cal{C}\longrightarrow \cal{A}_{CG}$ restricts to the functors $T_{CM}: \ ^{co}\cal{C}\longrightarrow\ ^{co}\cal{A}_{CG}$ and $T_{CM}\circ C:\cal{M}\longrightarrow \cal{A}_{coCG}$. The latter one further restricts to $T_{CM}\circ C:\cal{S}\longrightarrow \cal{A}_{coCG}^{co}$ by Theorem 1.4(1). The left vertical functor is obtained in Theorem 1.4(3).
\item[(2)] It is obvious from the definitions.
\item[(3)] It is shown that the left vertical functor is an equivalence of categories in Theorem 1.4(3). The other three functors are equivalences of categories from Theorem A and Berstein's results mentioned in the proof of part 1.
\item[(4)] It is immediate from Theorem 1.4.
\end{itemize}
\vspace{-4mm}
\end{proof}
 Now we prove Theorem B, which is used to deduce Theorem 1.10 from Theorems 1.3 and 1.4.
 \begin{proof}[{\it{Proof of Theorem B}}]
 Part 1 and the assertion on $\nu$ in part 2 follow from [6, Corollary 3.2]. By the definition ([6, p. 261]), the underlying Hopf algebra of $H_{\ast}(\Omega X; R)$ is just the homology Hopf algebra of the $H$-group $\Omega X$ (cf. the proof of Corollary 3.2 of [6]). Hence the assertion on $\chi$ in part 2 follows. Part 3 follows from results in [9, pp. 68-69 and p. 75] and Theorem A(1).
 \end{proof}
\begin{proof}[\it{Proof of Theorem \ref{nusamequaltochi}.}]
Let the conditions (i)$'$, (ii)$'$ and (iii)$'$ denote the conditions (i), (ii) and (iii) in Theorem \ref{nuequaltochi} applied to $A=H_{\ast}(\Omega X;\bb{Q})$, and let the condition (iv)$'$ denote the condition (iv) in the proof of Theorem 1.4(2) applied to $A=H_{\ast}(\Omega X;\bb{Q})$. Then it is enough to show that the conditions (i)-(iv) are equivalent to the conditions (i)$'$-(iv)$'$ respectively, and that (iv) is equivalent to (v).
\par\indent We can easily see that (i)-(iii) imply (i)$'$-(iii)$'$ respectively from Theorem B. Recall that $\Omega X$ is homotopy equivalent to the (weak) product of rational Eilenberg-MacLane complexes ([9, pp. 69-70]). Then we also see the converses.
\par\indent For the equivalence (iv)$\Leftrightarrow$(iv)$'$, recall that $S(H_{\ast}(\Omega\Sigma Y;\bb{Q}))\cong H_{\ast}(Y;\bb{Q})$ (cf. the proof of [6, Corollary 1.4]). Then the implication (iv)$\Rightarrow$(iv)$'$ follows from Theorem A. The converse follows from Theorem B(3).
\par\indent
We end the proof by showing the equivalence (iv)$\Leftrightarrow$(v). The implication (iv)$\Rightarrow$(v) is obvious. For the converse, it is enough to see that if $X$ is rationally equivalent to $S^{2n+1}$, $X$ is co-$H$-equivalent to $\Sigma S_{(0)}^{2n}$. Since $H_{\ast}(\Omega X;\bb{Q})\cong H_{\ast}(\Omega \Sigma S^{2n};\bb{Q})\cong T\overline{H}_{\ast}(S^{2n};\bb{Q})$ holds in $\cal{A}$, the underlying algebras of $H_{\ast}(\Omega X;\bb{Q})$ and $H_{\ast}(\Omega \Sigma S^{2n};\bb{Q})$ are isomorphic to the polynomial ring $\bb{Q}[z]$ on a generator $z$ of degree $2n$. Thus Corollary 1.5 implies that both $H_{\ast}(\Omega X;\bb{Q})$ and $H_{\ast}(\Omega \Sigma S^{2n};\bb{Q})$ are isomorphic to $A(2n)$ in $\cal{A}_{CG}$, and hence that $X$ is co-$H$-equivalent to $\Sigma S_{(0)}^{2n}$, by Theorem B(3).
\end{proof} 
\subsection*{Acknowledgment}
My heartfelt appreciation goes to Prof. Arkowitz whose comments and suggestions were of inestimable value for my study. The main results in this paper resulted from an attempt to answer questions raised by him. I am also indebt to Prof. Bergman whose comments made enormous contribution to my work.

\end{document}